\documentclass[12pt]{amsart}
\usepackage{enumerate}

\usepackage[breaklinks=true,colorlinks=true,linkcolor=green,citecolor=red,urlcolor=blue,pdfencoding=auto]{hyperref}

\allowdisplaybreaks

\theoremstyle{plain}
\newtheorem*{theorem}{Theorem}
\newtheorem{lemma}{Lemma}
\newtheorem{proposition}{Proposition}

\newcommand\im{\operatorname{im}}

\newcommand\Span{\operatorname{Span}}

\begin{document}

\title[Nontrivial central representation]{A class of nilpotent Lie algebras whose center acts nontrivially in cohomology}

\author{Grant Cairns \and Barry Jessup \and Yuri Nikolayevsky}
\address{Department of Mathematics and Statistics, La Trobe University, Melbourne, 3086, Australia}
   \email{G.Cairns@latrobe.edu.au}
\address{Department of Mathematics and Statistics, University of Ottawa, Ottawa, K1N 6N5, Canada} 
\email{Barry.Jessup@uottawa.ca}
\address{Department of Mathematics and Statistics, La Trobe University, Melbourne, 3086, Australia}
   \email{Y.Nikolayevsky@latrobe.edu.au}



\thanks{This research was supported in part by NSERC and in part,by ARC Discovery grant DP130103485}

\begin{abstract}
We show that the central representation is nontrivial for all one-dimensional central extensions of nilpotent Lie algebras possessing a codimension one abelian ideal. 
\end{abstract}

\maketitle
\section{Introduction}
\label{s:intro}

We consider finite dimensional Lie algebras $L$ over $\mathbb{R}$. The cohomology ring $H^*(L)$ with trivial coefficients is naturally a module over the centre $Z$ of $L$; for each $z\in Z$ and $[a]\in H^k(L)$, the class $z\cdot [a]$ is defined as $[i_z a]\in H^{k-1}(L)$, where $i_z$ denotes the interior product by $z$. This action of $Z$ on $H^*(L)$ extends to an action of the exterior algebra $\Lambda Z$ called the \emph{central representation}. In \cite{CJ1} we conjectured that the central representation is nontrivial for all nilpotent Lie algebras. This conjecture was established in \cite{CJ1} for several classes of algebras, and in \cite{SP}, for $2$-step nilpotent algebras (on the other hand, a non-nilpotent Lie algebra for which the central representation is trivial was given in \cite{CJ1}). Examples where the central representation is faithful were given in \cite{CJ1, CJ2}. The free $2$-step nilpotent Lie algebras on more than two generators provide examples where the central representation is not faithful \cite{CJ2}. The aim of this present paper is to establish the above conjecture for a class of nilpotent algebras of higher nilpotency obtained by a natural extension of abelian algebras.

There are two classic inductive constructions for building nilpotent Lie algebras; each uses a nilpotent Lie algebra $L$ to build a nilpotent Lie algebra $L'$ with $\dim L'=\dim L+1$. In the first construction, as studied by Dixmier \cite{Dix} for example, one takes a nilpotent derivation $D$ of $L$,  introduces a new generator $u$ and defines a Lie algebra structure on $L'=L\oplus \mathbb{R}u$ having $L$ as an ideal, by setting $[u,x]:=D x$ for all $x\in L$. In the other construction, one obtains $L'$ as a central extension. To do this, choose a closed $2$-form $\Omega$ in $\Lambda^2L^*$, introduce a new generator $z$ and set $L'=L\oplus \mathbb{R}z$ where $z$ is taken to be central element with $L'/\mathbb{R}z\cong L$ and $[x,y]:=[x,y]_L+\Omega(x,y)z$ for all $x,y\in L$. The two constructions may be regarded as building $L'$ from the ``outside'' and the ``inside'' respectively; obviously, every nilpotent Lie algebra can be obtained from an abelian algebra by repeated applications of either of the above constructions. In this paper we examine Lie algebras that be built from abelian algebras by employing one construction of each type. Note that the resulting class of algebras does not depend on which construction we apply first. We also note that the repeated double extension construction, starting from an abelian algebra, naturally appears in the classification of bi-invariant pseudo-Riemannian homogeneous manifolds \cite{MR}.

We prove the following.
\begin{theorem}
The central representation is nontrivial for all one-dimensional central extensions of nilpotent Lie algebras possessing a codimension one abelian ideal.
\end{theorem}

Note that the Theorem remains valid for Lie algebras over $\mathbb{C}$ (with no changes to the proof). We also note that the non-triviality of the central representation for Lie algebras obtained from an abelian algebra by just one extension (of either type) trivially follows.

\section{Preliminaries}
\label{s:pre}

\subsection{Linear-algebraic reduction} 
\label{ss:reform}

Consider a finite dimensional vector space $V$ over $\mathbb{R}$ and a nilpotent linear map $\theta: V \to V$. In order to simplify the notation, we write $\omega_1 \omega_2$ for $\omega_1 \wedge \omega_2$ throughout the paper. Extend $\theta$ to a derivation of $\Lambda V$ which we still denote $\theta$ (so that $\theta(\omega_1 \omega_2)=\theta(\omega_1) \omega_2+\omega_1 \theta(\omega_2)$, for all $\omega_1, \omega_2 \in \Lambda V$). For $\Omega \in \Lambda^2 V$, denote $\mu_{\Omega}:\Lambda V\to \Lambda V$ the right multiplication by $\Omega$; that is, $\mu_{\Omega}:\omega \mapsto \omega \Omega$.

The proof of the Theorem is based on the following Proposition.
\begin{proposition}\label{p:suffice}
In the notation above, for all $\epsilon\in V$ and $\Omega\in\Lambda^2V$ such that $\Omega\in \ker \theta$ and $\Omega \not\in \im \theta$, there exists $\beta\in \Lambda V$ such that
\begin{enumerate}[{\rm (A)}]
    \item \label{it:bkermu}
    $\Omega\beta = 0$,

    \item \label{it:bnotimmu}
    $\beta \not\in \im \mu_\Omega$,

    \item \label{it:bkertheta}
    $\theta\beta = 0$,

    \item \label{it:alpha}
    there exist $\alpha, \gamma \in \Lambda V$ such that $\epsilon\beta + \Omega \alpha = \theta \gamma$.
\end{enumerate}
\end{proposition}

\begin{proof}[Proof of the Theorem assuming Proposition~\ref{p:suffice}]
Let a Lie algebra $L$ be defined as a one-dimensional central extension of a Lie algebra $W$ which has an abelian ideal $I$ of codimension $1$. We will prove that the interior multiplication by $z$ is nontrivial in $H^*(L)$, where  $L=W \oplus \mathbb{R}z$. Denote $z^* \in L^*$ a non-zero form such that $z^*(W)=0$, and denote $u^* \in W^*$ a non-zero form such that $u^*(I)=0$. Note that $dz^* \in \Lambda^2 W^*$ and we can write $dz^*=u^* \epsilon + \Omega$, where $\epsilon \in I^*, \, \Omega \in \Lambda^2 I^*$. Furthermore, $du^*=0$ and $d \phi= u^* \theta \phi$ for $\phi \in \Lambda I^*$. Note that we necessarily have $\theta \Omega = 0$.

We want to construct $\omega = z^*(u^* \alpha + \beta) + u^* \delta + \gamma$, where $\alpha, \beta, \gamma, \delta \in \Lambda I^*$, such that $\omega$ is closed, but $u^* \alpha + \beta$ is not exact. The form $\omega$ is closed if and only if conditions~\eqref{it:bkermu}, \eqref{it:bkertheta} and \eqref{it:alpha} are simultaneously satisfied. The fact that $u^* \alpha + \beta = d(z^*(u^* \phi + \psi) + u^* \eta + \rho)$ for some $\phi, \psi, \eta, \rho \in \Lambda I^*$ is equivalent to following three equations:
\begin{equation} \label{eq:notd}
\Omega \psi = \beta, \qquad \theta \psi = 0, \qquad \Omega \phi + \epsilon \psi + \theta \rho = \alpha.
\end{equation}

Now if $\Omega \in \im \theta$, say $\Omega = \theta \delta$, for some $\delta \in \Lambda I^*$, we may set $\alpha = 1, \beta = 0$, and $\gamma = - \delta$. Then conditions~\eqref{it:bkermu}, \eqref{it:bkertheta}, \eqref{it:alpha} are satisfied, but the last equation in~\eqref{eq:notd} is not, for any choice of $\phi, \psi, \rho \in \Lambda I^*$. We can therefore assume that $\Omega \in \im \theta$, and then by condition~\eqref{it:bnotimmu} the first equation in~\eqref{eq:notd} can never be satisfied. Thus the Theorem follows from Proposition~\ref{p:suffice} for $V=I^*$.
\end{proof}

The proof of Proposition~\ref{p:suffice} which we give in Section~\ref{s:proof} requires some preparation.

\subsection{Lefschetz Property and canonical forms of \texorpdfstring{$\Omega$}{\textOmega} and \texorpdfstring{$\theta$}{\texttheta}}
\label{ss:weak}

Let $V, \theta$ and $\Omega$ be as in the assumptions of Proposition~\ref{p:suffice}. The rank $r$ of $\Omega$ is defined to be the maximal number $k$ such that $\Omega^k \ne 0$ (note that $r \ge 1$ as $\Omega \not\in \im \theta$). Then $\Omega^r = v_1v_2 \dots v_{2r}$ for some linear independent $v_1, \dots, v_{2r} \in V$. This decomposition is not unique, but the subspace $S= \Span(v_1, v_2, \dots, v_{2r}) \subset V$ called the \emph{support} of $\Omega$ does not depend on a particular choice of the decomposition. We clearly have $\Omega \in \Lambda^2 S$. Furthermore, from the fact that $\theta \Omega = 0$ it follows that both $S$ and $\ker \mu_\Omega$ are $\theta$-invariant.

We will need the following fact.

\begin{lemma}[\textbf{Multilinear Lefschetz Property}] \label{l:Lefschetz}
{\ }

\begin{enumerate}[{\rm (a)}]
  \item \label{it:Lis}
  The map $\mu_\Omega: \Lambda^k S \to \Lambda^{k+2}S$ is injective for $0 \le k \le r-1$, and is surjective for $r + 1 \le k + 2 \le 2r$.

  \item \label{it:Lbij} 
  For $0 \le k \le r$, the map $\mu^k_\Omega : \Lambda^{r-k}S \to \Lambda^{r+k}S$ is a linear isomorphism.
\end{enumerate}
\end{lemma}
Note that \eqref{it:Lis} follows from \eqref{it:Lbij} by the dimension count; \eqref{it:Lbij} is well known (see e.g. \cite[Proposition~1.2.30]{H}) and may be considered as an easy version of the Hard Lefschetz theorem in complex geometry, while an elementary proof of the finite characteristic $p$ version of \eqref{it:Lis} is given in \cite{CJa} and the characteristic zero result then follows by letting $p$ tend to infinity.

The following fact will be used in the proof of Proposition~\ref{p:suffice} to deduce condition~\eqref{it:bnotimmu} from condition~\eqref{it:bkermu}. Let $T$ be an (arbitrary) a linear complement of $S$ in $V$.

\begin{lemma} \label{l:Omegabeta}
Suppose $\beta \in \Lambda^{\ge r}S \otimes \Lambda T$ has a non-zero summand, say $\beta_r$, in $\Lambda^r S \otimes \Lambda T$. If $\Omega\beta = 0$, then $\beta \not\in \im \mu_\Omega$.
\end{lemma}
\begin{proof}
Write $\beta = \beta_r + \beta_{>r}$ with $\beta_{>r} \in \Lambda^{>r}S \otimes \Lambda T$, so that $\beta_r \ne 0$. Suppose $\Omega\beta = 0$, but $\beta \in \im \mu_\Omega$. As $\Omega \in \Lambda^2 S$, this implies $\Omega\beta_r = 0$ and $\beta_r \in \im \mu_\Omega$. We have $\beta_r = \Omega(
\sum_{i=1}^p \sigma_i \otimes \omega_i)$, where $\sigma_i \in  \Lambda^{r-2}S$ and where $\omega_i \in \Lambda T$ are linear independent. Then
$0 = \Omega \beta_r = \sum_{i=1}^p (\Omega^2\sigma_i) \otimes \omega_i$, 
and so $\Omega^2 \sigma_i = 0$, for all $i=1, \dots, p$. Then by Lemma~\ref{l:Lefschetz}\eqref{it:Lbij} with $k = 2$, we obtain $\sigma_i = 0$ for all $i=1, \dots, p$, and so $\beta_r = 0$, a contradiction.
\end{proof}

\smallskip

Another ingredient of the proof is the following canonical form for the restrictions of $\Omega$ and $\theta$ to $S$. Note that $S$ is $\theta$-invariant. Moreover, relative to a basis for $S$, the matrix of $\Omega$ is symplectic and the fact that $\theta \Omega = 0$ means that the matrix of the restriction of $\theta$ on $S$ is a (nilpotent) Hamiltonian matrix.

\begin{lemma}[{\cite[Theorem~9]{LM}}] \label{l:can}
There exists a direct sum decomposition $S=\oplus_{a=1}^p (U^a \oplus V^a)$ $\oplus \oplus_{b=1}^q Z^b$ such that $p, q \ge 0, \, p+q > 0$, and
\begin{enumerate}[{\rm (1)}]
  \item
  $\dim U^a=\dim V^a = 2l_a+1, \, l_a \ge 0$, for all $a=1, \dots, p$, and $\dim Z^b = 2m_b, \, m_b \ge 1$, for all $b=1, \dots, q$.

  \item
  For all $a=1, \dots, p, \; b=1, \dots, q$, there exist bases $\{u^a_{i_a}\}$ for $U^a$, $\{v^a_{i_a}\}$ for $V^a$ and $\{z^b_{j_b}\}$ for $Z^b$, such that
  \begin{enumerate}[{\rm (a)}]
    \item
    $\theta u^a_1=\theta v^a_1=\theta z^b_1=0$ and $\theta u^a_{i_a}=u^a_{i_a-1}, \theta v^a_{i_a} = v^a_{i_a-1}$ for $2 \le i_a \le 2l_a+1$ and $\theta z^b_{j_b}=z^b_{j_b-1}$ for $2 \le j_b \le 2m_b$.
    \item
    The $2$-vector $\Omega$ is given by 
    \begin{multline*}
      \Omega= \sum_{a=1}^p (u^a_{2l_a+1} v^a_1 - u^a_{2l_a} v^a_2 + \dots + u^a_1 v^a_{2l_a+1}) \\
      + \sum_{b=1}^q c_b(z^b_{2m_b}z_1-z^b_{2m_b-1}z_2+ \dots + (-1)^{m_b+1}z^b_{m_b+1}z_{m_b}),
    \end{multline*}
    where $c_b=\pm 1$.
  \end{enumerate}
\end{enumerate}
\end{lemma}

\section{Proof of Proposition~\ref{p:suffice}}
\label{s:proof}

In the assumptions and notation of Proposition~\ref{p:suffice} we choose the direct decomposition of the support $S$ of $\Omega$ and the corresponding bases in the subspaces of that decomposition as in Lemma~\ref{l:can}.

For a set $P$ of nonzero vectors $(r_a, s_a) \in \mathbb{R}^2, \; a=1, \dots, q$, define
\begin{equation} \label{eq:betaP}
\beta_P=\prod_{a=1}^p \Big((r_au^a_{l_a+1}+s_av^a_{l_a+1})\prod_{i_a=1}^{l_a}u^a_{i_a} \prod_{i_a=1}^{l_a}v^a_{i_a} \Big) \prod_{b=1}^q \Big(\prod_{j_b=1}^{m_b} z^b_{j_b}\Big).
\end{equation}
Denote $S_P$ \emph{the support of $\beta_P$}, the linear span of $x \in V$ such that $\beta_P x = 0$. Clearly $S_P \subset S$ and $\theta S_P \subset S_P$, for any choice of the set $P$.

Note that by Lemma~\ref{l:can}, for any $P$, the element $\beta_P \in \Lambda^r S \subset \Lambda V$ defined by \eqref{eq:betaP} satisfies \eqref{it:bkertheta} and \eqref{it:bkermu}, and then also \eqref{it:bnotimmu}, by Lemma~\ref{l:Omegabeta}. The main difficulty is to satisfy~\eqref{it:alpha}. In the trivial case $\epsilon = 0$, we take $\beta=\beta_P$, with any $P$, and $\alpha=0,\; \gamma = 0$. In the following we assume $\epsilon \ne 0$.


We start with two easy cases.

\begin{lemma} \label{l:easy}
{\ }

\begin{enumerate}[{\rm (1)}]
  \item \label{it:easyN} 
  Let $N \ge 1$ be such that $\theta^{N-1} \epsilon \not \in S$ and $\theta^{N} \epsilon \in S_P$, for some choice of $P$ \emph{(}it may occur that $\theta^{N} \epsilon =0$\emph{)}. Then $\beta=\epsilon (\theta \epsilon) \dots (\theta^{N-1} \epsilon) \beta_P$ satisfies conditions \emph{(}\ref{it:bkermu}--\ref{it:alpha}\emph{)}.

  \item \label{it:easyS}
  If $\epsilon \in S$ then $\beta= \beta_P$  satisfies conditions \emph{(}\ref{it:bkermu}--\ref{it:alpha}\emph{)} \emph{(}with an arbitrary choice of $P$\emph{)}.
\end{enumerate}
\end{lemma}

\begin{proof}
  For assertion \eqref{it:easyN}, conditions (\ref{it:alpha}, \ref{it:bkermu}) and \eqref{it:bkertheta} are trivially satisfied, and then \eqref{it:bnotimmu} follows from Lemma~\ref{l:Omegabeta}.

  For \eqref{it:easyS}, the only condition to check is \eqref{it:alpha}. It is satisfied because $\epsilon \beta \in \im \mu_\Omega$ which follows from Lemma~\ref{l:Lefschetz}\eqref{it:Lbij} with $k=1$.
\end{proof}

By Lemma~\ref{l:easy} we can now assume that if $N$ is the smallest number for which $\theta^{N} \epsilon \in S$, then $N \ge 1$, and moreover, $\theta^{N} \epsilon \not\in S_P$, for any choice of $P$. Let $M > N$ be the smallest number for which there exists $P$ such that $\xi:=\theta^M \epsilon \in S_P$. Note that $\xi \ne 0$. Indeed, if it were so, the vector $\theta^{M-1}\epsilon$ would be a non-zero element of $S$ and we would have $\theta^{M-1}\epsilon = \sum_{a=1}^{p}(c_a u^a_1 + d_a v^a_1)+ \sum_{b=1}^{q} f_b z^b_1$ for some $c_a, d_a, f_b \in \mathbb{R}$, not all zeros. But then $\theta^{M-1}\epsilon \in S_P$ if we choose the elements of $P$ in such a way that $(r_a,s_a)=(c_a,d_a)$ when the latter vector is non-zero and $(r_a,s_a)=(1,0)$ otherwise; this contradicts the choice of $M$.

We can decompose $\xi \in S_P$ into the ``top" and the ``bottom" components, $\xi=\xi_T + \xi_B$, where $\xi_T \in \Span_{a=1}^p (r_au^a_{l_a+1}+s_av^a_{l_a+1}) \oplus \Span_{b=1}^q (z^b_{m_b})$ and $\xi_B \in \Span (u^a_{i_a}, v^a_{i_a}, z^b_{j_b} \, | \, i_a \le l_a,$ $j_b < m_b)$. Note that the ``top" component $\xi_T$ must be non-zero since $\theta^{M-1}\epsilon \in S \setminus S_P$.

\medskip

We consider several cases.

\underline{Case 1.} If $M$ is even, we are done. Indeed, choose $\beta=\beta_P$ (with the set $P$ used to define $M$) and let $w$ be a vector from the set $\{r_au^a_{l_a+1}+s_av^a_{l_a+1}, z^b_{m_b} \, | \, a =1, \dots, p, \; b =1, \dots, q\}$ whose coefficient in $\xi_T$ is non-zero. Then $\beta_P= c w \sigma$, where $\sigma$ is the product of all the vectors on the right-hand side of the formula \eqref{eq:betaP} for $\beta_P$ except for $w$, and $c \in \mathbb{R} \setminus \{0\}$. We have $\theta \sigma=0$ and $\beta_P=c (\theta^M \epsilon) \sigma$, and so
\begin{equation*}
c^{-1} \epsilon \beta_P = \epsilon (\theta^M \epsilon) \sigma= \theta\Big(\big(\epsilon (\theta^{M-1} \epsilon) - (\theta\epsilon) (\theta^{M-2} \epsilon) + \dots +(-1)^{M/2-1}(\theta^{M/2-1}\epsilon) (\theta^{M/2} \epsilon)\big) \sigma\Big).
\end{equation*}
Therefore $\epsilon \beta_P \in \im \theta$, hence condition \eqref{it:alpha} is satisfied.

\smallskip

\underline{Case 2.} Now consider the case when $M$ is odd. Note that by construction, $M > N \ge 1$, and so $M \ge 3$. We again take the decomposition $\xi=\xi_T+\xi_B$ as above and choose $w$ to be one of the ``top" vectors, as in the previous paragraph.

\underline{Case 2.1.} First suppose that $\theta w \ne 0$. We take $\beta=\beta_P$ and will prove that $\epsilon \beta \in \im \theta$. The proof is similar to the above, but more technical. For some non-zero $c \in \mathbb{R}$ we have $\beta_P =c w (\theta w) \sigma$, where $\sigma$ is defined as follows. If $w=z^b_{j_b}$, then $\sigma$ is the product of all the vectors on the right-hand side of \eqref{eq:betaP} except for $z^b_{j_b}$ and $z^b_{j_b-1}$. If $w= r_au^a_{l_a+1}+s_av^a_{l_a+1}$, then $\sigma$ is the product of all the vectors on the right-hand side of \eqref{eq:betaP} except for $r_au^a_{l_a+1}+s_av^a_{l_a+1}$ and $u^a_{l_a}$ if $r_a \ne 0$, and except for $r_au^a_{l_a+1}+s_av^a_{l_a+1}$ and $v^a_{l_a}$ if $r_a = 0$ (note that then $s_a \ne 0$). Similar to the above, we have $\theta \sigma = 0$ and $\epsilon \beta_P = c \epsilon (\theta^M \epsilon) (\theta^{M+1} \epsilon) \sigma$. Moreover, $(\theta^2 \xi) \sigma = 0$, that is, $(\theta^{M+2} \epsilon) \sigma = 0$. To prove that $\epsilon \beta_P \in \im \theta$ we denote $K=\frac12(M-1)$ and define
\begin{equation*}
\delta=\epsilon (\theta^{K} \rho) - (\theta \epsilon) (\theta^{K-1} \rho) + \dots + (-1)^{K} (\theta^{K}\epsilon) \rho, \quad \text{where }\rho= (\theta^M \epsilon) (\theta^{K+1} \epsilon).
\end{equation*}
Then we have
\begin{equation*}
  \theta(\delta)= \epsilon (\theta^{K+1} \rho) + (-1)^{K} (\theta^{K+1}\epsilon) \rho = \epsilon (\theta^{K+1} \rho),
\end{equation*}
and so $\theta(\delta \sigma)= \epsilon (\theta^{K+1} \rho) \sigma = \epsilon \theta^{K+1} ((\theta^M \epsilon) (\theta^{K+1} \epsilon)) \sigma$. But
$(\theta^{M+2} \epsilon) \sigma = 0$, and so $(\theta^r \epsilon) \sigma = 0$, for all $r \ge M+2$. It follows that $\epsilon \theta^{K+1} ((\theta^M \epsilon) (\theta^{K+1} \epsilon)) \sigma = K \epsilon (\theta^{M+1} \epsilon) (\theta^M \epsilon) \sigma$, and so $\theta(\delta \sigma)= -c^{-1}K \epsilon \beta_P$, as required for~\eqref{it:alpha}.

\smallskip

\underline{Case 2.2.} We now assume that $M$ is odd (recall that $M \ge 3$) and that for all the ``top" elements $w$ in the decomposition of $\xi_T$ we have $\theta w = 0$. This means that $\xi_T$ is a nonzero linear combination of some of the $r_au^a_{l_a+1}+s_av^a_{l_a+1}$ with $l_a=0$ and some of the $z^b_{m_b}$ with $m_b=1$. Recall that $\xi = \theta^M \epsilon$, and that for $N <M$ we have $\theta^N \epsilon \in S$. Therefore $\xi \in \theta S$ which implies that $\xi$ (and $\xi_T$) contain no terms $r_au^a_{l_a+1}+s_av^a_{l_a+1}$ with $l_a=0$. Then $\xi_T= \sum_{b: m_b=1} c_b z^b_1$, where $c_b \in \mathbb{R}$ and at least one of $c_b$ is non-zero. Up to relabelling we can take $c_1 \ne 0$. First assume that either $q > 1$ or there exists $1 \le a \le p$ such that $l_a \ne 0$. We again take $\beta=\beta_P$ and prove that $\epsilon \beta \in \im \theta$. Denote $\sigma$ the product of all the vectors on the right-hand side of \eqref{eq:betaP} except for $z^b_1$. Then $\theta \sigma = 0$. Define an element $\tau$ as follows. If $q>1$, replace the term $z_{m_b}^b$ in $\sigma$ by $z_{m_b+1}^b$ (note that $z_{m_b+1}^b \in S$). If $q=1$, but $l_a > 0$ for some $a=1, \dots, p$, replace the term $r_au^a_{l_a+1}+s_av^a_{l_a+1}$ in $\sigma$ by $r_au^a_{l_a+2}+s_av^a_{l_a+2}$ (note that $r_au^a_{l_a+2}+s_av^a_{l_a+2} \in S$). The resulting element $\tau$ contains no $z^1_j$ and has the property that $\theta \tau = \sigma$. Note that $\theta^M \epsilon (= \xi) = c_1z^1_1 +\sum_{b: m_b=1, b \ne 1} c_b z^b_1 + \phi$, where $\phi$ is a linear combination of the ``lower terms", $u^a_{i_a}$ and $v^a_{i_a}$ with $i_a \le l_a$ and $z^b_{j_b}$ with $j_b < m_b$. It follows that
\begin{equation} \label{eq:sigmatau}
  (\theta^M \epsilon) \sigma = c_1 \beta, \qquad (\theta^{M+1} \epsilon) \tau = 0.
\end{equation}

To prove that $\epsilon \beta_P \in \im \theta$ we denote $K=\frac12(M-1) \ge 1$ and define
\begin{multline*}
\delta=K\epsilon (\theta^{M-1} \epsilon) - (K-1)(\theta \epsilon) (\theta^{M-2} \epsilon) + \dots \\ + (-1)^{K-2} 2 (\theta^{K-2}\epsilon) (\theta^{K+2}\epsilon) + (-1)^{K-1} (\theta^{K-1}\epsilon) (\theta^{K+1}\epsilon),
\end{multline*}
Then we have
\begin{equation*}
  \theta(\delta)= (K+1)\epsilon (\theta^{M} \epsilon) - \rho, \quad \text{where} \; \; \rho= \epsilon (\theta^{M} \epsilon) - (\theta \epsilon) (\theta^{M-1} \epsilon) + \dots + (-1)^{K-1} (\theta^{K}\epsilon) (\theta^{K+1}\epsilon).
\end{equation*}
Note that $\theta \rho = \epsilon \theta^{M+1}\epsilon$ and so from \eqref{eq:sigmatau} we obtain
\begin{equation*}
  \theta(\delta \sigma + \rho \tau)= ((K+1)\epsilon (\theta^{M} \epsilon) - \rho)\sigma + \epsilon (\theta^{M+1}\epsilon) \tau + \rho \sigma = c_1 (K+1) \beta,
\end{equation*}
as required for~\eqref{it:alpha}.

\smallskip

\underline{Case 2.3.} In this last remaining case we have $q=1$ and $m_1=1$, and $l_a=0$, for all $a=1, \dots, p$, so that $S = \Span(u^1_1, v^1_1, \dots, u^p_1, v^p_1, z^1_2, z^1_1)$. In what follows we drop the subscript $1$ in $u^a_1$ and $v^a_1$ and the superscript $1$ in $z^1_2$ and $z^1_1$. From Lemma~\ref{l:can} we have
\begin{equation*}
  \Omega = z_2 z_1 + \sum_{a=1}^{p} u^av^a
\end{equation*}
(up to the sign). Furthermore, for some $N>0$ and some odd $M>N$ we have
\begin{equation*}
\theta^N \epsilon \in S, \qquad \theta^{N-1} \epsilon \not\in S, \qquad \theta^M \epsilon = z_1
\end{equation*}
(up to multiplying $\epsilon$ by a non-zero number). It follows that $M=N+1$ and that
\begin{equation*}
 \theta^{M-1} \epsilon = z_2 + c z_1 + \phi, \quad \text{where} \; \; \phi = \sum (r_a u_a + s_a v_a),
\end{equation*}
for some $c, r_a, s_a \in \mathbb{R}$. We can also assume that $p > 0$ as otherwise $\Omega = \theta ((\theta^{M-2} \epsilon) z_1)$ (contradicting the assumption that $\Omega \not\in \im \theta$).

Our construction for $\beta$ will be different from what we had before. Denote $S'= \Span(u^1, v^1, \dots, u^p, v^p)$ and $\Sigma=\sum_{a=1}^{p} u^av^a \in \Lambda^2 S'$, so that $\Omega=z_2z_1+\Sigma$. Now define
\begin{equation*}
  \beta=(z_2z_1- \Sigma) \lambda + z_1 \nu,
\end{equation*}
where
\begin{equation}\label{eq:lambdanu}
\lambda \in \Lambda^{p-1}S' \setminus \{0\}, \; \Sigma^2 \lambda = 0, \qquad \nu \in \Lambda^{p}S', \; \Sigma \nu = 0
\end{equation}
(the proof of existence of such elements and their concrete choice we postpone to a little later). As $\theta z_2 = z_1$ and $\theta z_1= \theta u^a=\theta u^b=0$, condition~\eqref{it:bkertheta} is satisfied for our $\beta$. Condition~\eqref{it:bkermu} follows from~\eqref{eq:lambdanu}, and then \eqref{it:bnotimmu} follows by Lemma~\ref{l:Omegabeta}. Furthermore, taking $\alpha= \epsilon \lambda$ we have
\begin{align*}
  \epsilon \beta + \Omega \alpha &= \epsilon (2z_2z_1 - \Omega) \lambda + \epsilon z_1 \nu + \Omega \epsilon \lambda = 2\epsilon (\theta^{M-1} \epsilon - \phi)z_1 \lambda + \epsilon z_1 \nu \\
  &=  2\epsilon (\theta^{M-1} \epsilon) (\theta^M \epsilon) \lambda + \epsilon z_1 (2\phi \lambda + \nu).
\end{align*}
But $\epsilon (\theta^{M-1} \epsilon) (\theta^M \epsilon) \lambda \in \im \theta$: using the fact that $\theta^{M+1} \epsilon = \theta \lambda=0$ (and by  calculations similar to those in Case 2.1) we can check that $\epsilon (\theta^{M-1} \epsilon) (\theta^M \epsilon) \lambda = \theta (\delta \lambda)$, where
\begin{equation*}
\delta = \sum_{i=0}^{K-1}(-1)^{i+1} (\theta^{K-1-i}\rho) (\theta^i\epsilon) \quad \text{and} \; \; K=\frac12(M-1), \; \rho = (\theta^M \epsilon)(\theta^K \epsilon).
\end{equation*}
So with our choice of $\beta$ and $\alpha$, condition~\eqref{it:alpha} will be satisfied provided $\nu=-2\phi \lambda$. Substituting this into~\eqref{eq:lambdanu} we obtain that to conclude the proof we have to construct a non-zero $\lambda \in \Lambda^{p-1}S'$ such that $\Sigma^2 \lambda = \Sigma \phi \lambda = 0$. Note that multiplication $\mu_\Sigma$ by $\Sigma$ is a linear isomorphism from $\Lambda^{p-1}S'$ to $\Lambda^{p+1}S'$ by Lemma~\ref{l:Lefschetz}\eqref{it:Lbij}, so it is sufficient to find a non-zero $\eta (=\Sigma \lambda) \in \Lambda^{p+1}S'$ such that $\Sigma \eta = \phi \eta = 0$.

If $p=1$ we take $\eta = \Sigma$. Let $p>1$. If $\phi = 0$ we use the fact that $\mu_\Sigma$ is a surjective map from $\Lambda^{p+1}S'$ to $\Lambda^{p+3}S'$ by Lemma~\ref{l:Lefschetz}\eqref{it:Lis}. Comparing the dimensions we find that it has a nontrivial kernel, so there exists a non-zero $\eta \in \Lambda^{p+1}S'$ such that $\Sigma \eta = 0$. If $\phi \ne 0$, we take $\eta = \phi \zeta$ with $\zeta \in \Lambda^{p}S'$, where $\zeta \not\in \phi \Lambda^{p-1}S'$ and $\Sigma \zeta \in \phi \Lambda^{p+1}S'$. But now $\dim (\phi \Lambda^{p-1}S') = \dim \Lambda^{p-1}(S'/\phi) = \binom{2p-1}{p-1}$ and, as $\mu_\Sigma:\Lambda^{p}S' \to \Lambda^{p+2}S'$ is surjective by Lemma~\ref{l:Lefschetz}\eqref{it:Lis}, we have
\begin{multline*}
\dim \{\zeta \in \Lambda^{p}S' \, | \, \Sigma \zeta \in \phi \Lambda^{p+1}S'\} = \dim \ker(\mu_\Sigma:\Lambda^{p}S' \to \Lambda^{p+2}S') + \dim (\phi \Lambda^{p+1}S')\\
= \dim \Lambda^{p}S' - \dim \Lambda^{p+2}S' + \dim \Lambda^{p+1}(S'/\phi)= \binom{2p}{p}-\binom{2p}{p+2}+ \binom{2p-1}{p+1}.
\end{multline*}
So $\dim \{\zeta \in \Lambda^{p}S' \, | \, \Sigma \zeta \in \phi \Lambda^{p+1}S'\}-\dim (\phi \Lambda^{p-1}S')=\frac{3}{p+2}\binom{2p}{p+1} > 0$ concluding the proof.

\bigskip
\noindent \emph{Acknowledgement.} The first author would like to thank the members of the Department of Mathematics and Statistics at the University of Ottawa for their hospitality during his stay there. The second author would like to thank the members of the Department of Mathematics and Statistics at La Trobe University for their very considerate hospitality during his many visits there.

\bibliographystyle{amsplain}
\bibliography{props3}

\end{document}